\documentclass[11pt]{amsart}
\usepackage{mathrsfs}
\usepackage{amssymb}
\usepackage{verbatim}
\usepackage{hyperref}
\usepackage{color}
\usepackage{amsfonts}
\usepackage{mathrsfs}
\usepackage{amsmath}
\usepackage{amssymb}

\begin{document}
\newtheorem{theorem}{Theorem}
\newtheorem{proposition}[theorem]{Proposition}
\newtheorem{claim}[theorem]{Claim}
\newtheorem{conjecture}[theorem]{Conjecture}
\def\theconjecture{\unskip}
\newtheorem{corollary}[theorem]{Corollary}
\newtheorem{lemma}[theorem]{Lemma}
\newtheorem{sublemma}[theorem]{Sublemma}
\newtheorem{observation}[theorem]{Observation}
\theoremstyle{definition}
\newtheorem{definition}{Definition}
\newtheorem{notation}[definition]{Notation}
\newtheorem{remark}[definition]{Remark}
\newtheorem{question}[definition]{Question}
\newtheorem{questions}[definition]{Questions}
\newtheorem{example}[definition]{Example}
\newtheorem{problem}[definition]{Problem}
\newtheorem{exercise}[definition]{Exercise}

\numberwithin{theorem}{section} \numberwithin{definition}{section}
\numberwithin{equation}{section}

\def\earrow{{\mathbf e}}
\def\rarrow{{\mathbf r}}
\def\uarrow{{\mathbf u}}
\def\varrow{{\mathbf V}}
\def\tpar{T_{\rm par}}
\def\apar{A_{\rm par}}

\def\reals{{\mathbb R}}
\def\torus{{\mathbb T}}
\def\heis{{\mathbb H}}
\def\integers{{\mathbb Z}}
\def\naturals{{\mathbb N}}
\def\complex{{\mathbb C}\/}
\def\distance{\operatorname{distance}\,}
\def\support{\operatorname{support}\,}
\def\dist{\operatorname{dist}\,}
\def\Span{\operatorname{span}\,}
\def\degree{\operatorname{degree}\,}
\def\kernel{\operatorname{kernel}\,}
\def\dim{\operatorname{dim}\,}
\def\codim{\operatorname{codim}}
\def\trace{\operatorname{trace\,}}
\def\Span{\operatorname{span}\,}
\def\dimension{\operatorname{dimension}\,}
\def\codimension{\operatorname{codimension}\,}
\def\nullspace{\scriptk}
\def\kernel{\operatorname{Ker}}
\def\ZZ{ {\mathbb Z} }
\def\p{\partial}
\def\rp{{ ^{-1} }}
\def\Re{\operatorname{Re\,} }
\def\Im{\operatorname{Im\,} }
\def\ov{\overline}
\def\eps{\varepsilon}
\def\lt{L^2}
\def\diver{\operatorname{div}}
\def\curl{\operatorname{curl}}
\def\etta{\eta}
\newcommand{\norm}[1]{ \|  #1 \|}
\def\expect{\mathbb E}
\def\bull{$\bullet$\ }

\def\xone{x_1}
\def\xtwo{x_2}
\def\xq{x_2+x_1^2}
\newcommand{\abr}[1]{ \langle  #1 \rangle}

\newcommand{\Norm}[1]{ \left\|  #1 \right\| }
\newcommand{\set}[1]{ \left\{ #1 \right\} }
\def\one{\mathbf 1}
\def\whole{\mathbf V}
\newcommand{\modulo}[2]{[#1]_{#2}}
\def \essinf{\mathop{\rm essinf}}
\def\scriptf{{\mathcal F}}
\def\scriptg{{\mathcal G}}
\def\scriptm{{\mathcal M}}
\def\scriptb{{\mathcal B}}
\def\scriptc{{\mathcal C}}
\def\scriptt{{\mathcal T}}
\def\scripti{{\mathcal I}}
\def\scripte{{\mathcal E}}
\def\scriptv{{\mathcal V}}
\def\scriptw{{\mathcal W}}
\def\scriptu{{\mathcal U}}
\def\scriptS{{\mathcal S}}
\def\scripta{{\mathcal A}}
\def\scriptr{{\mathcal R}}
\def\scripto{{\mathcal O}}
\def\scripth{{\mathcal H}}
\def\scriptd{{\mathcal D}}
\def\scriptl{{\mathcal L}}
\def\scriptn{{\mathcal N}}
\def\scriptp{{\mathcal P}}
\def\scriptk{{\mathcal K}}
\def\frakv{{\mathfrak V}}
\def\C{\mathbb{C}}
\def\R{\mathbb{R}}
\def\Rn{{\mathbb{R}^n}}
\def\Sn{{{S}^{n-1}}}
\def\M{\mathbb{M}}
\def\N{\mathbb{N}}
\def\Q{{\mathbb{Q}}}
\def\Z{\mathbb{Z}}
\def\F{\mathcal{F}}
\def\L{\mathcal{L}}
\def\S{\mathcal{S}}
\def\supp{\operatorname{supp}}
\def\dist{\operatorname{dist}}
\def\essi{\operatornamewithlimits{ess\,inf}}
\def\esss{\operatornamewithlimits{ess\,sup}}
\author{Mingming Cao}
\address{Mingming Cao
\\
School of Mathematical Sciences
\\
Beijing Normal University
\\
Laboratory of Mathematics and Complex Systems
\\
Ministry of Education
\\
Beijing 100875
\\
People's Republic of China
}
\email{m.cao@mail.bnu.edu.cn}

\author{Qingying Xue}
\address{
        Qingying Xue\\
        School of Mathematical Sciences\\
        Beijing Normal University \\
        Laboratory of Mathematics and Complex Systems\\
        Ministry of Education\\
        Beijing 100875\\
        People's Republic of China}
\email{qyxue@bnu.edu.cn}
\author{K\^{o}z\^{o} Yabuta}
\address{K\^{o}z\^{o} Yabuta\\Research center for Mathematical
Science \\Kwansei Gakuin University\\Gakuen 2-1, Sanda 669-1337\\
Japan }
\thanks{The second author was supported partly by NSFC
(No. 11471041), the Fundamental Research Funds for the Central Universities (No. 2014KJJCA10) and NCET-13-0065. The third named author was supported partly by Grant-in-Aid for Scientific Research (C) Nr. 15K04942, Japan Society
for the Promotion of Science.\\ \indent Corresponding
author: Qingying Xue\indent Email: qyxue@bnu.edu.cn}


\date{December 28, 2015.}
\title[Multilinear Fractional Strong Maximal Operator]{On the Boundedness of Multilinear Fractional Strong Maximal Operator with multiple weights}
\maketitle

\begin{abstract}In this paper, we investigated the boundedness of multilinear fractional strong maximal operator $\mathcal{M}_{\mathcal{R},\alpha}$ associated with rectangles or related to more general basis with multiple weights $A_{(\vec{p},q),\mathcal{R}}$. In the rectangles setting, we first gave an end-point estimate of $\mathcal{M}_{\mathcal{R},\alpha}$, which not only extended the famous linear result of Jessen, Marcinkiewicz and Zygmund, but also extended the multilinear result of Grafakos, Liu, P\'{e}rez and Torres ($\alpha=0$) to the case $0<\alpha<mn.$ Then, in one weight case, we gave several equivalent characterizations between $\mathcal{M}_{\mathcal{R},\alpha}$ and $A_{(\vec{p},q),\mathcal{R}}$, by applying a different approach from what we have used before. Moreover, a sufficient condition for the two weighted norm inequality of $\mathcal{M}_{\mathcal{R},\alpha}$ was presented and a version of vector-valued two weighted inequality for the strong maximal operator was established when $m=1$. In the general basis setting, we further studied the properties of the multiple weights $A_{(\vec{p},q),\mathcal{R}}$ conditions, including the equivalent characterizations and monotonic properties, which essentially extended one's previous understanding. Finally, a survey on multiple strong Muckenhoupt weights was given, which
demonstrates the properties of multiple weights related to rectangles systematically.
\end{abstract}


\section{Introduction}
It is well-known that the study of multi-parameter operators originated in the works of Fefferman and Stein \cite{F-S} on bi-parameter singular integral operators. Subsequently, Journ\'{e} \cite{Journe} gave a multi-parameter version of $T1$ theorem on product spaces. Later on, a new type of $T1$ theorem on product spaces was formulated by Pott and Villarroya \cite{PV}. Recently, Martikainen \cite{M2012} demonstrated a bi-parameter representation of singular integrals in expression of the dyadic shifts, which extended the famous result of Hyt\"{o}nen \cite{H} for one-parameter case. Still more recently, using the probabilistic methods and the techniques of dyadic analysis, Hyt\"{o}nen and Martikainen \cite{HM} gave a bi-parameter version of $T1$ theorem in spaces of non-homogeneous type. Furthermore, a bi-parameter version of $Tb$ theorem on product Lebesgue spaces was obtained by Ou \cite{Ou}, where $b$ is a tensor product of two pseudo-accretive functions.

It is also well-known that the most prototypical representative of the multi-parameter operators is the following strong maximal operator $M_{\mathcal{R}}$:
$$
M_{\mathcal{R}}f(x) := \sup_{\substack{R \ni x \\ R \in \mathcal{R} }} \frac{1}{|R|} \int_R |f(y)| dy, \ \ x \in \Rn,
$$
where $\mathcal{R}$ is the collection of all rectangles $R \subset \Rn$ with sides parallel to the coordinate axes.
It can be looked as a geometric maximal operator which commutes with full n-parameter group of dilations $(x_1,\ldots,x_n) \rightarrow (\delta_1 x_1,\ldots, \delta_n x_n)$. The strong $L^p(\Rn)(1<p<\infty)$ boundedness of $M_{\mathcal{R}}$ was given by Garc\'{i}a-Cuerva and Rubio de Francia \cite[~p.456]{C-Rubio}. In \rm{1935}, a maximal theorem was given by Jessen, Marcinkiewicz and Zygmund in \cite{JMZ}. They pointed out that unlike the classical Hardy-Littlewood maximal operator, the strong maximal function is not of weak type $(1,1)$. Moreover, they studied the end-point behavior of $M_{\mathcal{R}}$ and obtained the following inequality:
\begin{equation}\label{endpoint-JMZ}
\big|\{x \in \Rn; M_{\mathcal{R}}f(x)>\lambda \}\big|
\lesssim_{n} \int_{\Rn} \frac{|f(x)|}{\lambda} \left(1 + \Big(\log^+ \frac{|f(x)|}{\lambda}\Big)^{n-1}\right) dx.
\end{equation}
In 1975, C\'{o}rboda and Fefferman \cite{CF} gave a geometric proof of $(\ref{endpoint-JMZ})$ and established a covering lemma for rectangles. Their covering lemma is quite useful by the reason that it overcomes the failure of the Besicovitch covering argument for rectangles with arbitrary eccentricities. The selection algorithm given by C\'{o}rboda and Fefferman was used many times to gain end-point estimates for $M_{\mathcal{R}}$, as demonstrated in \cite{C}, \cite{F1981}, \cite{GLPT}, \cite{Ha-P}, \cite{L-Luque}, \cite{LS}, \cite{LP}, \cite{M}.

The corresponding weighted version of $(\ref{endpoint-JMZ})$ with $w \in A_{1,\mathcal{R}}$ was shown by Bagby and Kurtz \cite{BK}. In addition, the weighted weak type and strong type norm inequalities for vector-valued strong maximal operator were obtained in \cite{CG}. It is worth pointing out that it is the first time to avoid using C\'{o}rboda-Fefferman's covering lemma to obtain the end-point estimate of $M_{\mathcal{R}}$. Subsequently, the above weighted results were improved by enlarging the range of weights class  in \cite{LP} and \cite{M}. In \cite{LP}, Luque and Parissis formulated a weighted version of C\'{o}rdoba-Fefferman covering lemma and showed that the following weighted inequality holds:
\begin{equation}\label{Weighted}
w\big(\{x \in \Rn; M_{\mathcal{R}}f(x)>\lambda \}\big)
\lesssim_{w,n} \int_{\Rn} \frac{|f(x)|}{\lambda} \left(1 + \Big(\log^+ \frac{|f(x)|}{\lambda}\Big)^{n-1}\right) M_{\mathcal{R}} w(x) dx.
\end{equation}
For $n=2$, inequality (\ref{Weighted}) was proved by Mitsis \cite{M} whenever $w \in A_{p,\mathcal{R}}$ and $1<p<\infty$. Unfortunately, the combinatorics of two-dimensional rectangles that the author used are not available in higher dimensions. To overcome this obstacle, Luque and Parissis \cite{LP} adopted a different approach, which relies heavily on the best constant of the weighted estimates of the strong maximal operator \cite{LS}.

In 2011, Grafakos et al. \cite{GLPT} first introduced the multilinear version of the strong maximal operator $\mathcal{M}_{\mathcal{R}}$ by setting
\begin{equation}\label{M-R}
\mathcal{M}_{\mathcal{R}}(\vec{f})(x)=
\sup_{\substack{R \ni x \\ R\in\mathcal{R}}} \prod_{i=1}^m \frac{1}{|R|} \int_R |f_i(y)| dy, \ \  \quad x \in \Rn
\end{equation}
where $\vec{f}=(f_1, \cdots, f_m)$ is an $m$-dimensional vector of locally integrable functions.
It was shown that for any $\lambda > 0$, the following inequality holds
\begin{equation}\label{Endpoint-Grafakos}
\Big| \big\{x \in \Rn; \mathcal{M}_{\mathcal{R}}(\vec{f})(x) > \lambda^m \big\} \Big|
\lesssim_{m,n} \bigg(\prod_{i=1}^m  \int_{\Rn} \Phi_n^{(m)} \left(\frac{|f_i(y)|}{\lambda}\right) dy\bigg)^{1/m},
\end{equation}
where $\Phi_n(t):=t[1+(\log ^+ t)^{n-1}]$ $(t>0)$ and $\Phi_n^{(m)}$ is
$m$-times compositions of the function $\Phi_n$ with itself.
Furthermore, inequality (\ref{Endpoint-Grafakos}) is sharp in the sense that one cannot replace $\Phi_n^{(m)}$ by $\Phi_n^{(k)}$ for $k \leq m-1$. Similarly, one can define the multilinear maximal function $\mathcal{M}_{\mathscr{B}}$ on a general basis $\mathscr{B}$ if $\mathcal{R}$ is replaced by $\mathscr{B}$ in $(\ref{M-R})$. In \cite{GLPT} , the authors also proved that for a Muckenhoupt basis $\mathscr{B}$, the multilinear maximal operator
$\mathcal{M}_{\mathscr{B}}$ is bounded from $L^{p_1}(w_1) \times \cdots \times L^{p_m}(w_m)$ to $L^{p,\infty}(v)$ provided that $(\vec{w},v)$ are weights satisfying $v \in A_{\infty,\mathscr{B}}$ and the power bump condition for some $r > 1$,
\begin{equation}\label{general baisis power bump}
\sup_{B \in \mathscr{B}} \left( \frac{1}{|B|} \int_B v dx \right)
\prod_{i=1}^m \left( \frac{1}{|B|} \int_B w_i^{(1-p_i')r} dx \right)^{{p}/{p_i' r}} < \infty.
\end{equation} Subsequently, under more weaker condition (Tauberian condition) than $v \in A_{\infty,\mathscr{B}}$, Liu and Luque \cite{L-Luque} investigated the strong boundedness of two-weighted inequality for the maximal operator $\mathcal{M}_{\mathscr{B}}$. They showed that if $M_{\mathscr{B}}$ satisfies the Tauberian condition (condition $(A)$ \cite{HLP}, \cite{J}, \cite{Perez1993}) with respect to some $\gamma \in (0,1)$ and a weight $\mu$ as follows: there exists a positive constant $C_{\mathscr{B},\gamma,\mu}$ such that, for all measurable sets $E$, it holds that
\begin{equation}\label{condition A}
\tag{$A_{\mathscr{B},\gamma,\mu}$}\mu \big(\{x \in \Rn; M_{\mathscr{B}}(\mathbf{1}_E)(x) > \gamma \} \big) \leq C_{\mathscr{B},\gamma,\mu}  \mu(E).
\end{equation}
 Then,
$\mathcal{M}_{\mathscr{B}}$ enjoys the boundedness property from the product spaces $L^{p_1}(w_1) \times \cdots \times L^{p_m}(w_m)$ to $ L^p(v)$.  Recently, Hagelstein et al. \cite{HLP} discussed the relationship between the boundedness of $M_{\mathscr{B}}$,  Tauberian condition $(A_{\mathscr{B},\gamma,\mu})$ and weighted Tauberian condition. Furthermore, Hagelstein and Parissis \cite{Ha-P} proved that the asymptotic estimate for weighted Tauberian constant associated to rectangles is equivalent to
$w \in A_{\infty,\mathcal{R}}$, which gives a new characterization of the class $A_{\infty,\mathcal{R}}$.

Still more recently, Cao, Xue and Yabuta \cite{CXY} introduced the multilinear fractional strong maximal operator $\mathcal{M}_{\mathcal{R}, \alpha}$ and multiple weights $A_{(\vec{p},q),\mathcal{R}}$ associated with rectangles as follows:
\begin{eqnarray*}
\mathcal{M}_{\mathcal{R},\alpha}(\vec{f})(x)&=&\sup_{\substack{R \ni x \\ R \in \mathcal{R}}}\prod_{i=1}^m \frac{1}{|R|^{1-\frac{\alpha}{mn}}} \int_R |f_i(y)| dy, \ \  \ \ \quad x\in \Rn, \\
{}[\vec{w},v]_{A_{({\vec{p}},q),\mathcal{R}}} &:=& \sup_{R \in \mathcal{R}} |R|^{\frac{\alpha}{n} + \frac1q - \frac1p}
\left( \frac{1}{|R|} \int_R \nu dx \right)^{\frac{1}{q}} \prod_{i=1}^m \left( \frac{1}{|R|} \int_R \omega_i^{1-p_i'} dx \right)^{\frac{1}{p_i'}} < \infty, \ \  0 \leq \alpha < mn.
\end{eqnarray*}
In order to establish the two-weighted estimates of $\mathcal{M}_{\mathcal{R},\alpha}$, the authors introduced the dyadic reverse doubling condition associated with rectangles, which is weaker than $A_{\infty,\mathcal{R}}$. It was showed that if each $w_i^{1-p_i'}$ satisfies the dyadic reverse doubling condition, then $\mathcal{M}_{\mathcal{R},\alpha}$ is bounded from $L^{p_1}(w_1) \times \cdots \times L^{p_m}(w_m)$ to $L^{q}(v)$ if and only if $(\vec{w},v) \in A_{({\vec{p}},q),\mathcal{R}}$.

Motivated by the works in \cite{CXY}, \cite{GLPT} and \cite{L-Luque}, in this paper, we will investigate the boundedness of multilinear strong and fractional strong maximal operators in the setting of rectangles and in the setting of more general basis. We are mainly concerned with the end-point behavior, characterizations of two weighted norm inequalities and vector-valued norm inequalities. A survey will also be given on multiple strong Muckenhoupt weights, which
demonstrates the properties of multiple weights associated with rectangles systematically.
\section{Definitions and Main results}

\subsection{Rectangle setting.} We now formulate the main results of the maximal operators related
to rectangles. The first result is concerned with the end-point behavior of $ \mathcal{M}_{\mathcal{R},\alpha}$.
\begin{theorem}\label{Theorem endpoint}
Let $n $, $m \geq 1$ and $0 \leq \alpha < mn$. Then for any $\lambda > 0$,
the following end-point inequality holds,
\begin{align*}
\Big| \big\{x \in \Rn; \mathcal{M}_{\mathcal{R},\alpha}(\vec{f})(x) > \lambda^m \big\} \Big|^{m-\frac{\alpha}{n}}
\lesssim_{m,n,\alpha}& \prod_{i=1}^m \bigg[ 1 + \bigg(\frac{\alpha}{mn}
\log^+ \prod_{j=1}^m \int_{\Rn} \Phi_n^{(m)} \Big(\frac{|f_j(y)|}{\lambda}\Big) dy \bigg)^{n-1}\bigg]^m \\
&\quad \times \int_{\Rn} \Phi_n^{(m)} \left(\frac{|f_i(y)|}{\lambda}\right) dy.
\end{align*}
\end{theorem}

\begin{remark}
If $m=1$ and $\alpha=0,$ then the above inequality in Theorem $\ref{Theorem endpoint}$ coincides with the inequality $(\ref{endpoint-JMZ})$. In the multilinear setting, if $\alpha=0$, Theorem $\ref{Theorem endpoint}$ recovers the inequality $(\ref{Endpoint-Grafakos})$. Therefore, Theorem \ref{Theorem endpoint} extends not only the linear result given by Jessen, Marcinkiewicz and Zygmund \cite{JMZ} but also extends the multilinear result obtained by Grafakos et. al. \cite{GLPT}. Even in the linear setting, Theorem $\ref{Theorem endpoint}$ is completely new for $0<\alpha<n$.
\end{remark}
In order to state the other results, we need to introduce one more definition.
\begin{definition}[\cite{L-Luque}]
Let $1<p<\infty$. A Young function $\Phi$ is said to satisfy the $B^*_p$
condition, written $\Phi \in B^*_p$, if there is a positive constant $c$
such that the following inequality holds
$$
\int_c^\infty \frac{\Phi_n(\Phi(t))}{t^p} \frac{dt}{t} < \infty,
$$
where $\Phi_n(t):=t[1+(\log ^+ t)^{n-1}]$ for all $t>0$.
\end{definition}
We obtain the two weighted, vector-valued estimate of $M_{\mathcal{R}}$ as follows:
\begin{theorem}\label{Theorem lq}
Let $1 < q < p < \infty$, $r=p/q$. Assume that $A$ and $B$ are Young functions
such that their complementary Young functions $\bar{A}$ and $\bar{B}$ satisfy
$\bar{A} \in B_{r'}^*$ and $\bar{B} \in B_q^*$, respectively. Let $(w,v)$ be
a couple of weights such that
\begin{equation}\label{A-B}
\sup_{R \in \mathcal{R}} \big\| w^q \big\|_{A,R}^{1/q}  \
\big\| v^{-1} \big\|_{B,R} < \infty.
\end{equation}
For some fixed $\gamma \in (0,1)$
and any nonnegative function $h \in L^{r'}(\Rn)$ with $||h||_{L^{r'}(\Rn)}=1$, assume that
$M_{\mathcal{R}}$ satisfies $(A_{\mathcal{R},\gamma,h})$ condition and
$(A_{\mathcal{R},\gamma,w^q h})$ condition.
Then, the following two weighted, vector-valued inequality holds for $M_{\mathcal{R}}$,
$$
\big\| M_{\mathcal{R}}f \big\|_{L^p(\ell^q,w^p)}
\lesssim \big\| f \big\|_{L^p(\ell^q, v^p)}.
$$
\end{theorem}
\begin{remark}
Theorem \ref{Theorem lq} was shown by
P\'{e}rez \cite{Perez1999}, whenever the family of rectangles $\mathcal{R}$ is replaced by cubes. Moreover, in the scalar-valued case,
it was proved by Liu and Luque \cite{L-Luque}.
\end{remark}
In order to establish the boundedness of the multilinear fractional strong
maximal operator $\mathcal{M}_{\mathcal{R},\alpha}$, we give the definition of
the corresponding multiple weights.
\begin{definition}[{Class of $A_{(\vec{p},q),\mathcal{R}}$}, \cite{CXY}]
Let$ 1 < p_1, \cdots, p_m < \infty$, $\frac{1}{p} = \frac{1}{p_{1}} + \cdots
+ \frac{1}{p_{m}}$, and $q > 0$.
Suppose that $\vec{w} = (w_1, \cdots, w_m)$ and each $w_i$ is a nonnegative
locally integrable function on $\Rn$.
We say that $\vec{w}$ satisfies the $A_{({\vec{p}},q),\mathcal{R}}$ condition
or $\vec{w} \in A_{(\vec{p},q),\mathcal{R}}$, if it satisfies
$$
[\vec{w}]_{A_{(\vec{p},q),\mathcal{R}}}
:= \sup_R \left( \frac{1}{|R|} \int_R \nu_{\vec{w}}^q dx \right)^{1/q}
\prod_{i=1}^m\left( \frac{1}{|R|} \int_R w_i^{-p_i'} dx \right)^{1/{p_i'}}
< \infty,
$$
where $\nu_{\vec{w}} = \prod_{i=1}^m w_i$. If $p_i = 1$,
$(\frac{1}{R} \int_R w_i^{1-p_i'})^{1/{p_i'}}$ is understood as
$(\inf_R w_i)^{-1}$.
\end{definition}
We formulate the weighted results of $ \mathcal{M}_{\mathcal{R},\alpha}$ in the following characterizations:
\begin{theorem}\label{Theorem one-weight}
Let $k \in \N$, $0 \leq \alpha < mn$, $\frac{1}{p} = \frac{1}{p_{1}} + \cdots + \frac{1}{p_{m}}$ with $1 < p_1, \cdots, p_m < \infty$,
and $0 < p \leq q < \infty$ satisfying $\frac{1}{q} = \frac{1}{p} - \frac{\alpha}{n}$.
Then the following statements are equivalent :
\begin{eqnarray}
\label{one-weight-R} &{}& \vec{w} \in A_{({\vec{p}},q),\mathcal{R}} ;\\
\label{one-weight-R-r} &{}&  \vec{w}^r \in A_{(\frac{\vec{p}}{r},\frac{q}{r}),\mathcal{R}}, \ for \ some \ r>1 ;\\
\label{M-R-a} &{}& \mathcal{M}_{\mathcal{R},\alpha} :  L^{p_1}(w_1^{p_1}) \times \cdots \times L^{p_m}(w_m^{p_m}) \rightarrow L^q(\nu_{\vec{w}}^q); \\
\label{M-R-a-Phi} &{}& \mathcal{M}_{\mathcal{R},\alpha,\Phi_{k+1}} : L^{p_1}(w_1^{p_1}) \times \cdots \times L^{p_m}(w_m^{p_m}) \rightarrow L^q(\nu_{\vec{w}}^q).
\end{eqnarray}
\end{theorem}
\begin{remark}
Although the fact that $(\ref {one-weight-R})$ is equivalent with $(\ref {M-R-a})$ was given in \cite{CXY}, we here present some new ingredients. In addition, Theorem \ref{Theorem one-weight} also tells us that the weights class $\vec{w} \in A_{({\vec{p}},q),\mathcal{R}}$ not only implies the boundedness of $\mathcal{M}_{\mathcal{R},\alpha}$ but also characterizes much more bigger operators $\mathcal{M}_{\mathcal{R},\alpha,\Phi_{k+1}}$.
\end{remark}
Further more, we obtain the following result:
\begin{theorem}\label{Theorem power-R}
Let $0 \leq \alpha < mn$, $\frac{1}{p}
= \frac{1}{p_{1}} + \cdots + \frac{1}{p_{m}}$ with
$1 < p_1, \cdots, p_m < \infty$,
and $0 < p \leq q < \infty$. If $(\vec{w},v)$ are weights such that
$v \in A_{\infty,\mathcal{R}}$ and the power bump condition holds for some
$ r > 1 $,
\begin{equation}\label{power condition}
\sup_{R \in \mathcal{R}} |R|^{\frac{\alpha}{n} + \frac{1}{q} -\frac{1}{p}}
\left( \frac{1}{|R|} \int_R v dx \right)^{\frac{1}{q}}
\prod_{i=1}^m \left( \frac{1}{|R|} \int_R w_i^{(1-p_i')r} dx
\right)^{\frac{1}{r p_i'}} < \infty,
\end{equation}
then $\mathcal{M}_{\mathcal{R},\alpha} : L^{p_1}(w_1) \times \cdots
\times L^{p_m}(w_m) \rightarrow L^q(v)$.
\end{theorem}
\begin{corollary}\label{Theorem M-two-weight}
Suppose that $0 \leq \alpha < mn$, $\frac{1}{p}
= \frac{1}{p_{1}} + \cdots + \frac{1}{p_{m}}$ with
$1 < p_1, \cdots, p_m < mn/{\alpha}$. Let each $u_i$ be nonnegative locally
integrable function.
Then $\vec{u} \in A_{\vec{p},\mathcal{R}}$ implies that
$$
\big\| \mathcal{M}_{\mathcal{R},\alpha}(\vec{f}) \big\|_{L^p(v^p)}
\leq C \prod_{i=1}^m \big\| f_i \big\|_{L^{p_i}(w_i^{p_i})} ,
$$
where $v=\prod_{i=1}^m u_i^{1/{p_i}}$ and $w_i=M_{\alpha p_i/m}(u_i)$.
\end{corollary}

\subsection{The general basis and two weight norm inequalities}\label{Sec-general}
In this subsection, we will present some general results for the maximal
operator defined on the general basis. We start by introducing some
definitions and notations, which will be used later.

By a basis $\mathscr{B}$ in $\Rn$ we mean a collection of open sets in $\Rn$.
We say that $w$ is a weight associated with the basis $\mathscr{B}$ if $w$ is
a nonnegative measurable function in $\Rn$ such that
$w(B) = \int_B w(x) dx < \infty$ for each $B \in \mathscr{B}$. Moreover, $w \in A_{p,\mathscr{B}}$ means that
$$
\sup_{B \in \mathscr{B}} \bigg( \frac{1}{|B|} \int_B w dx \bigg)
\bigg( \frac{1}{|B|} \int_B w^{1-p'} dx \bigg)^{p/{p'}}< \infty.
$$
We say that $\mathscr{B}$ is a Muckenhoupt basis if
$M_{\mathscr{B}} : L^p(w) \rightarrow L^p(w)
$ for any $1 < p < \infty$ and
for any $w \in A_{p,\mathscr{B}}$.

We also need some basic property of Orlicz spaces. More details can be found in \cite{RR}. A Young function is a continuous, convex, increasing function $\Phi : [0,\infty) \rightarrow [0,\infty)$ with $\Phi(0)=0$ and such that $\Phi(t)\rightarrow \infty$ as
$t \rightarrow \infty$. The $\Phi$-norm of a function $f$ over a set $E$ with finite measure is defined by
\begin{equation}\label{definition}
\big\| \Phi \big\|_{\Phi,E} = \inf \bigg\{\lambda>0; \frac{1}{|E|}\int_E \Phi \Big(\frac{|f(x)|}{\lambda}\Big) dx \leq 1 \bigg\}.
\end{equation}
For a given Young function $\Phi$, one can define a complementary function
$$
\bar{\Phi}(s)=\sup_{t>0} \{ st-\Phi(t)\}, \ \ s \geq 0.
$$
Moreover, the generalized H\"{o}lder inequality holds
\begin{equation}\label{Holder}
\frac{1}{|E|}\int_E |f(x)g(x)| dx \leq 2 \big\| f \big\|_{\Phi,E}\big\| f \big\|_{\bar{\Phi},E}.
\end{equation}

\begin{definition}\label{def-M-B-varphi-Psi}
Suppose that the function $\varphi:(0,\infty) \rightarrow (0,\infty)$ is essentially non-decreasing and
$\lim_{t \rightarrow \infty} \frac{\varphi(t)}{t}=0$. Assume that $\mathscr{B}$ is a basis and $\{ \Psi_i \}_{i=1}^m$ is a sequence of Young functions, we define the multilinear Orlicz maximal operator associated with the function $\varphi$ by
\begin{equation*}
\mathcal{M}_{\mathscr{B},\varphi,\overrightarrow{\Psi}}(\vec{f})(x)=
\sup_{\substack{B \ni x \\ B \in \mathscr{B}}} \varphi(|B|) \prod_{i=1}^m \big\| f_i \big\|_{\Psi_i, B}, \ \ x\in \Rn.
\end{equation*}
In particular, if $\Psi_i(t)=t$, $i=1,\ldots,m$, we denote $\mathcal{M}_{\mathscr{B},\varphi,\overrightarrow{\Psi}}$ by $\mathcal{M}_{\mathscr{B},\varphi}$. If $\varphi(t)=t^{\frac{\alpha}{n}}$, we denote $\mathcal{M}_{\mathscr{B},\varphi,\overrightarrow{\Psi}}$ and $\mathcal{M}_{\mathscr{B},\varphi}$ by $\mathcal{M}_{\mathscr{B},\alpha,\overrightarrow{\Psi}}$ and $\mathcal{M}_{\mathscr{B},\alpha}$ respectively.
When $\mathscr{B}=\mathcal{R}$, $\mathcal{M}_{\mathscr{B},\alpha}$ coincides with $\mathcal{M}_{\mathcal{R},\alpha}$.
\end{definition}
We summarize the main results as follow:
\begin{theorem}\label{Theorem M-B}
Let $0 < p \leq q < \infty$, $\frac{1}{p} = \frac{1}{p_{1}} + \cdots + \frac{1}{p_{m}}$ with $1 < p_1, \cdots, p_m < \infty$. Let $\mathcal{A}_i$, $\mathcal{B}_i$ and $\mathcal{C}_i$ $(i=1,\ldots,m)$ be Young functions such that
$\mathcal{A}_i^{-1}(t) \mathcal{C}_i^{-1}(t) \leq \mathcal{B}_i^{-1}(t)$, $t>0$ for each $i=1,\ldots,m$.
Assume that $\mathscr{B}$ is a basis and $\{ \mathcal{C}_i \}_{i=1}^m$ is a sequence of Young functions satisfying
$$ \mathcal{M}_{\mathscr{B},\overrightarrow{\mathcal{C}}} : L^{p_1}(\Rn) \times \cdots \times L^{p_m}(\Rn) \rightarrow L^p(\Rn).$$
If $(\vec{w},v)$ are weights such that $\mathcal{M}_{\mathscr{B},\varphi,\overrightarrow{\mathcal{B}}}$ satisfies  $(A_{\mathscr{B},\gamma,v^q})$ condition and
\begin{equation}\label{two-weight-Young condition}
\sup_{B\in \mathscr{B}} \varphi(|B|)|B|^{\frac{1}{q} -\frac{1}{p}} \left( \frac{1}{|B|} \int_B v^q dx \right)^{\frac{1}{q}}
\prod_{i=1}^m \big\| w_i^{-1} \big\|_{\mathcal{A}_i,B} < \infty,
\end{equation}
then $\mathcal{M}_{\mathscr{B},\varphi,\overrightarrow{\mathcal{B}}} : L^{p_1}(w_1^{p_1}) \times \cdots \times L^{p_m}(w_m^{p_m}) \rightarrow L^q(v^q)$.
\end{theorem}
\begin{corollary}\label{Theorem power-B}
Let $0 \leq \alpha < mn$, $\frac{1}{p} = \frac{1}{p_{1}} + \cdots + \frac{1}{p_{m}}$ with $1 < p_1, \cdots, p_m < \infty$,
and $0 < p \leq q < \infty$. Assume that $\mathscr{B}$ is a Muckenhoupt basis.
If $(\vec{w},v)$ are weights such that $\mathcal{M}_{\mathscr{B},\alpha}$ satisfies $(A_{\mathscr{B},\gamma,v})$ condition and the power bump condition
\begin{equation}\label{power bump condition}
\sup_{B\in \mathscr{B}} |B|^{\frac{\alpha}{n} + \frac{1}{q} -\frac{1}{p}} \left( \frac{1}{|B|} \int_B v dx \right)^{\frac{1}{q}}
\prod_{i=1}^m \left( \frac{1}{|B|} \int_B w_i^{(1-p_i')r} dx \right)^{\frac{1}{r p_i'}} < \infty, \quad \ for \ some \ r > 1,
\end{equation}
then $\mathcal{M}_{\mathscr{B},\alpha} : L^{p_1}(w_1) \times \cdots \times L^{p_m}(w_m) \rightarrow L^q(v)$.
\end{corollary}

\begin{remark}
It is easy to see that our Corollary 2.7 extends Theorem $2.3$ of Grafakos et al. \cite{GLPT} in the following sense. Under the same assumptions, the authors \cite{GLPT} only got the boundedness from $L^{p_1}(w_1) \times \cdots \times L^{p_m}(w_m)$ to $L^{p,\infty}(v)$. On the other hand, we enlarge the range of $\alpha$ from $\alpha=0$ to $0 \leq \alpha < mn$.
\end{remark}
Finally, we present a two weighted norm inequality in the more general context of Banach function spaces.
\begin{theorem}\label{Theorem M-Y}
Let $\frac{1}{p} = \frac{1}{p_{1}} + \cdots + \frac{1}{p_{m}}$ with $1 < p_1, \cdots, p_m < \infty$,
and $0 < p \leq q < \infty$. Let $\varphi$ be a function as in Definition $\ref{def-M-B-varphi-Psi}$. Suppose that $Y_1,\ldots,Y_m$ be Banach function spaces such that
$$ \mathcal{M}_{\vec{Y}'} : L^{p_1}(\Rn) \times \cdots \times L^{p_m}(\Rn)\rightarrow L^p(\Rn) .$$
If $(\vec{w},v)$ are weights such that $\mathcal{M}_{\vec{Y}'}$ satisfies $(A_{\mathscr{B},\gamma,v^q})$ condition  and
\begin{equation}\label{two-weight-Banach condition}
\sup_{B\in \mathscr{B}} \varphi(|B|)|B|^{\frac{1}{q} -\frac{1}{p}} \left( \frac{1}{|B|} \int_B v^q dx \right)^{\frac{1}{q}}
\prod_{i=1}^m \big\| w_i^{-1} \big\|_{Y_i,B} < \infty,
\end{equation}
then $\mathcal{M}_{\mathscr{B},\varphi} : L^{p_1}(w_1^{p_1}) \times \cdots \times L^{p_m}(w_m^{p_m}) \rightarrow L^q(v^q)$.
\end{theorem}

This article is organized as follows: A survey on multiple strong Muckenhoupt weights $A_{(\vec{p},q),\mathcal{R}}$ will be given in Section $\ref{Sec-weights}$. Section $\ref{Sec-1-5}$ will be devoted to give the proof of Theorem $\ref{Theorem endpoint}$ and Theorem $\ref{Theorem one-weight}$. In Section $\ref{Sec-Last}$, we will complete the proofs of the rest Theorems.

\section{A Survey on Multiple Strong Muckenhoupt Weights}\label{Sec-weights}
In this section, our goal is to study the properties of multiple weights related to rectangles systematically. We first recall the definition of $A_{\vec{p},\mathcal{R}}$ which was introduced in \cite{GLPT}.
\begin{definition}[\textbf{Multiple weights} $A_{\vec{p},\mathcal{R}}$, \cite{GLPT}]
Let $1 \leq p_1,\ldots,p_m < \infty$. We say that $m$-tuple of weights $\vec{w}$ satisfies the $A_{\vec{p},\mathcal{R}}$ condition ( or $\vec{w} \in A_{\vec{p},\mathcal{R}}$ ) if
$$
[\vec{w}]_{A_{\vec{p},\mathcal{R}}} := \sup_{R \in \mathcal{R}} \left(\frac{1}{|R|}\int_R \widehat{\nu}_{\vec{w}} dx \right)
\prod_{i=1}^m \left(\frac{1}{|R|} \int_R w_i^{1-p_i'} dx \right)^{{p}/{p_i'}} < \infty,
$$
where $\widehat{\nu}_{\vec{w}}=\prod_{i=1}^m w_i^{p/{p_i}}$. If $p_i = 1$, $(\frac{1}{R} \int_R w_i^{1-p_i'})^{1/{p_i'}}$ is understood as $(\inf_R w_i)^{-1}$.
\end{definition}

The characterizations of multiple weights are as follows.

\begin{theorem}\label{weight-1}
Let $1 \leq p_1, \cdots, p_m < \infty$, $\frac{1}{p} = \frac{1}{p_{1}} + \cdots + \frac{1}{p_{m}}$ and $p_0=\min\{p_i\}_i$. Then the following statements hold :
\begin{enumerate}
\item [(1)] $ A_{r_1 \vec{p},\mathcal{R}} \subsetneqq A_{r_2 \vec{p},\mathcal{R}}, \ \ \text{for any} \ \ 1/p_0 \leq r_1 < r_2 < \infty.$
\item [(2)] $A_{\vec{p},\mathcal{R}}=\bigcup_{1/p_0 \leq r < 1} A_{r \vec{p},\mathcal{R}}$.
\item [(3)] $\vec{w} \in A_{{\vec{p}},\mathcal{R}}$ if and only if
$$
{\widehat{\nu}_{\vec{w}}} \in A_{mp,\mathcal{R}} \ \ \text{and} \ \ w_i^{1-p_i'} \in A_{mp_i',\mathcal{R}}, \ i=1,\ldots,m,
$$
where $w_i^{1-p_i'} \in A_{mp_i',\mathcal{R}}$ is understood as $w_i^{1/m} \in A_{1,\mathcal{R}}$ if $p_i=1$.
\end{enumerate}
\end{theorem}

\begin{theorem}\label{weight-2}
Let $1 \leq p_1, \cdots, p_m < \infty$, $\frac{1}{p} = \frac{1}{p_{1}} + \cdots + \frac{1}{p_{m}}$
and $\frac{1}{m}\leq p \leq q < \infty$. There holds
\begin{enumerate}
\item [(i)] $\vec{w} \in A_{({\vec{p}},q),\mathcal{R}}$ if and only if
$$
{\nu_{\vec{w} }}^q  \subset A_{mq,\mathcal{R}} \ \ \text{and} \ \ w_i^{-p_i'} \in A_{mp_i',\mathcal{R}}, \ i=1,\ldots,m.
$$
When $p_i = 1$, $w_i^{-p_i'}$ is understood as $w_i^{1/m} \in A_{1,\mathcal{R}}$.
\item [(ii)] Assume that $0 < \alpha < mn$, $p_1,\ldots,p_m < \frac{mn}{\alpha}$ and $\frac{1}{q} = \frac{1}{p} - \frac{\alpha}{n}$. Then $\vec{w} \in A_{({\vec{p}},q),\mathcal{R}}$ if and only if
$$
{\nu_{\vec{w} }}^q \in A_{q(m-\alpha/n),\mathcal{R}} \ \ \text{and} \ \ w_i^{-p_i'} \in A_{p_i'(m-\alpha/n),\mathcal{R}}, \ i=1,\ldots,m.
$$
When $p_i = 1$, $w_i^{-p_i'} \in A_{p_i'(m-\alpha/n),\mathcal{R}}$ is understood as $w_i^{n/(mn-\alpha)} \in A_{1,\mathcal{R}}$.
\end{enumerate}
\end{theorem}

\begin{theorem}\label{weight-3}
Let $1 < p_1, \cdots, p_m < \infty$, $\frac{1}{p} = \frac{1}{p_{1}} + \cdots + \frac{1}{p_{m}}$, $\frac{1}{m} < p \leq q < \infty$ and
$p_0 = \min \{ p_i \}_i$. There holds that
\begin{enumerate}
\item [(a)] $ A_{(\vec{p},q,r_2),\mathcal{R}} \subsetneqq A_{(\vec{p},q,r_1),\mathcal{R}}, \ \text{whenever} \ \
              1 \leq r_1 < r_2 < p_0 $.
\item [(b)] For any $1\leq r_1 < p_0$,
$$ A_{(\vec{p},q,r_1),\mathcal{R}}=\bigcup_{r_1< r < p_0} A_{(\vec{p},q,r),\mathcal{R}},$$
where
$A_{(\vec{p},q,s),\mathcal{R}}:= \Big\{\vec{w}; \vec{w}^s=(w_1^s,\ldots,w_m^s) \in A_{(\frac{\vec{p}}{s},\frac{q}{s}),\mathcal{R}} \Big\}, \ \ \text{for any} \ s \geq 1$.
\end{enumerate}
\end{theorem}

\vspace{0.3cm}
\noindent\textbf{Proofs of Theorems 3.1-3.3.}
The argument used in Theorem 2.4 and Theorem 3.11 in \cite{CWX} relies only on the use of H\"{o}lder's inequality, and it doesn't involve any geometric property of cubes or rectangles. Hence we may also use the methods in \cite{CWX} to complete our proof. Since the main ideas are almost the same, we omit the proof here. It is worth mentioning that when considering the strict inclusion relationship in Theorem $\ref{weight-1}$ $(1)$ and Theorem $\ref{weight-3}$ $(a)$, we need the characterization of $|x|^\alpha \in A_{p,\mathcal{R}}$, which will be shown in Proposition $\ref{x-alpha}$ below.

\qed

The notation $\mathcal{DR}$ will always denote the family of all dyadic rectangles in $\Rn$ with sides parallel to the axes.
Now, we recall the definition of the dyadic reverse doubling condition.
\begin{definition}[\textbf{Dyadic reverse doubling condition}, \cite{CXY}]
We say that a nonnegative measurable function $\omega$ satisfies the dyadic reverse doubling condition,
or $\omega \in RD^{(d)}$, if $\omega$ is locally integrable on $\Rn$ and there is a constant $d>1$ such that
$$ d \int_I \omega(x)dx \leq \int_J \omega(x)dx $$
for any $I,J \in \mathcal{DR}$, where $I \subset J$ and $|I| = \frac{1}{2^n} |J|$.
\end{definition}
\begin{proposition}
There holds that
$A_{\infty,\mathcal{R}}(\Rn) \subsetneqq RD^{(d)}(\Rn)$, for any $n \geq 2$.

\end{proposition}
\begin{proof}
The inclusion relationship $A_{\infty,\mathcal{R}}(\Rn) \subset RD^{(d)}(\Rn)$ has
 been proved in Proposition 4.2 \cite{CXY}.
Thus, it suffices to show that there exists some weight
$w \in RD^{(d)}(\Rn) \setminus \bigcup_{1\le p<\infty} A_{p,\mathcal R}$.
In fact, let
$$
w(x_1,\dots,x_{n-1},x_n)=\frac{1}{(1+|x_n|)^2}, \
w_1(x_1,\dots,x_{n-1})=1, \
w_n(x_n)=\frac{1}{(1+|x_n|)^2}.
$$
Then, for any $I,J\in \mathcal D\mathcal R$, where $I\subset J$ and has half
side length of $J$, we have
\begin{equation*}
w(I)=\frac{1}{2^{n-1}}w_1(J_1\times\cdots\times J_{n-1})w_n(I_n),
\end{equation*}
where $I=I_1\times\cdots\times I_{n-1}\times I_n$,
$J=J_1\times\cdots\times J_{n-1}\times J_n$.
Hence
\begin{equation*}
{2^{n-1}}w(I)\le w(J).
\end{equation*}
This shows that $w\in RD^{(d)}(\Rn)$ for any $1<d\le 2^{n-1}$.

Now, we are in the position to show that $w\notin A_{\infty,\mathcal R}(\Rn)$.
Denote
$$
R_\ell :=[0,2^\ell)^n \ \text{and} \ E_\ell
:= [0,2^\ell)^{(n-1)}\times[0,1), \ \text{for each} \ \ell \in \N_+.
$$
Then, we have
\begin{equation*}
w(R_\ell)=2^{\ell(n-1)}\int_{0}^{2^\ell}\frac{dt}{(1+t)^2}
=2^{\ell( n-1)}(1-\frac{1}{1+2^\ell})=\frac{2^{\ell n}}{1+2^\ell}.
\end{equation*}
Therefore,
\begin{equation*}
\frac{w(E_\ell)}{w(R_\ell)}
=2^{\ell(n-1)}\cdot\frac12 \Big/ \frac{2^{\ell n}}{1+2^\ell}
=\frac12(1+\frac{1}{2^\ell}),
\end{equation*}
and
\begin{equation*}
\frac{|E_\ell|}{|R_\ell|}=\frac{1}{2^\ell}.
\end{equation*}
Hence there exists no $C,\delta>0$ such that
\begin{equation*}
\frac{w(E_\ell)}{w(R_\ell)}
\leq C \left(\frac{|E_\ell|}{|R_\ell|}\right)^{\delta},\quad \quad \ell \in \N_+.
\end{equation*}
which implies $w\notin A_{\infty,\mathcal R}(\Rn)$.

\end{proof}

\begin{proposition}\label{x-alpha}
Let $1 < p < \infty$. The strong Muckenhoupt weight has the characterization:
$|x|^\alpha\in A_{p,\mathcal R}(\Rn)$ if and only if $-1<\alpha<p-1$.
\end{proposition}
Although this proposition is contained in \cite{Kurtz}, we here present
a new proof.
\begin{proof}
This follows from Lemma 2.2 in Kurtz \cite[p.~239]{Kurtz}, and the following
fact:
\begin{equation}\label{eq:(1+|x|)a-Ap}
w(t)=(1+|t|)^\alpha\in A_p(\R)\ \text{ if and only if } -1<\alpha<p-1.
\end{equation}
In the case $-1<\alpha\le 0$, we see that $t^\alpha\in A_1(\R_+)$ and is
decreasing. So, $|x|^\alpha\in \widetilde A_1(\R_+)$, and hence by Theorem
4.4 in \cite{Yabuta} it belongs to
$A_{1,\mathcal R}(\Rn)\subset A_{p,\mathcal R}(\Rn)$.

\noindent
In the case $0<\alpha<p-1$, we have $-1<\alpha/(1-p)<0$, and so
$t^{\alpha/(1-p)}\in A_1(\R_+)$ and is decreasing. Hence
$|x|^\alpha=(|x|^{\alpha/(1-p)})^{1-p}\in \widetilde A_p(\R_+)$, and so, as
before, it belongs to $A_{p,\mathcal R}(\Rn)$.

Here, $$\widetilde A_p(\mathbb R_+):=\{
\omega(x)=\nu_1(|x|)\nu_2(|x|)^{1-p};\ \nu_1,\nu_2\in A_1(\mathbb R_+)\
\text{ are decreasing or}\ \nu_1^2,\nu_2^2\in A_1(\mathbb R_+)\}$$ and
$$\tilde A_1(\mathbb R_+):=\{
\omega(x)=\nu_1(|x|);\ \nu_1\in A_1(\mathbb R_+)\
\text{ is decreasing or}\ \nu_1^2\in A_1(\mathbb R_+)\},$$ which are
 the weight classes introduced by Duoandikoetxea \cite{Duo}.
\end{proof}

\section{Proofs of Theorem $\ref{Theorem endpoint}$ and Theorem \ref{Theorem one-weight}}\label{Sec-1-5}

To show the endpoint estimate of $\mathcal{M}_{\mathcal{R},\alpha}$, we need the following key lemma.
\begin{lemma}[\cite{GLPT}]\label{Phi-Phi-m}
Let $m \in \N$, and $E$ be any set. Let $\Phi$ be a submultiplicative Young function. If there is a constant $C$ such that
$$
1 < \prod_{i=1}^m ||f_i||_{\Phi,E}
$$
holds, then it yields that
$$
\prod_{i=1}^m ||f_i||_{\Phi,E} \leq C \prod_{i=1}^m \frac{1}{|E|}\int_{E} \Phi^{(m)}(|f_i(x)|)dx.
$$
\end{lemma}

\vspace{0.3cm}
\noindent\textbf{Proof of Theorem 2.1.}
Denote $E=\{x \in \Rn; \mathcal{M}_{\mathcal{R},\alpha} f(x) > \lambda^m \}$. Then there exists a compact set $K$ such that $K \subset E$ and
$$
|K| \leq |E| \leq 2|K|.
$$
By the compactness of $K$, one can find a finite collection of rectangles $\{R_j\}_{j=1}^N$ such that
\begin{equation}\label{covering}
K \subset \bigcup_{j=1}^N R_j \ \  \text{and} \ \ \lambda^m < \prod_{i=1}^m \frac{1}{|R_j|^{1-\frac{\alpha}{mn}}} \int_{R_j} |f_i(y)| dy, \ \ j=1,\ldots,N.
\end{equation}
According to the C\'{o}rdoba-Fefferman's rectangle covering lemma \cite{CF}, there are positive constants $\delta$, $c$ depending only on $n$, and a subfamily $\{\widetilde{R}_j\}_{j=1}^\ell$ of $\{R_j\}_{j=1}^N$ satisfying
\begin{equation}\label{subfamily}
\Big| \bigcup_{j=1}^N R_j \Big| \leq c \Big| \bigcup_{j=1}^\ell \widetilde{R}_j \Big|
\end{equation}
and
\begin{equation}\label{exp}
\int_{\bigcup_{j=1}^\ell \widetilde{R}_j} \exp \Big(\delta \sum_{j=1}^\ell \mathbf{1}_{\widetilde{R}_j}(x)\Big)^{\frac{1}{n-1}} dx
\leq 2 \Big| \bigcup_{j=1}^\ell \widetilde{R}_j \Big|.
\end{equation}
For convenience, we introduce the notations: $\widetilde{E} = \bigcup_{j=1}^\ell \widetilde{R}_j$ and $\Psi_n(t)=\exp(t^{\frac{1}{n-1}})-1$.
Then the inequality $(\ref{exp})$ is the same as
$$
\frac{1}{|\widetilde{E}|} \int_{\widetilde{E}} \Psi_n \Big(\delta \sum_{j=1}^\ell \mathbf{1}_{\widetilde{R}_j}(x)\Big) dx \leq 1.
$$
Furthermore, using the fact
\begin{equation}\label{equivalent}
\big\| f \big\|_{\Phi,E} \leq 1
\Leftrightarrow \frac{1}{|E|} \int_E \Phi(|f(x)|) dx \leq 1, \ \ \text{for any set} \ |E| < \infty,
\end{equation}
one can obtain
\begin{equation}\label{Psi}
\Big\| \sum_{j=1}^\ell \mathbf{1}_{\widetilde{R}_j} \Big\|_{\Psi_n,\widetilde{E}} \leq \delta^{-1}.
\end{equation}

Therefore, in all, combining the inequalities $(\ref{covering})$ and $(\ref{subfamily})$, we have
\begin{align*}
|\widetilde{E}|^{1-\frac{\alpha}{mn}}
&=\Big| \bigcup_{j=1}^\ell \widetilde{R}_j \Big|^{1-\frac{\alpha}{mn}} \\
&\leq \sum_{j=1}^{\ell} |\widetilde{R}_j|^{1-\frac{\alpha}{mn}} \bigg(\frac{1}{\lambda^m}\prod_{i=1}^m \frac{1}{|\widetilde{R}_j|^{1-\frac{\alpha}{mn}}} \int_{\widetilde{R}_j}|f_i(y)|dy \bigg)^{1/m} \\
&= \sum_{j=1}^{\ell} \bigg(\prod_{i=1}^m \int_{\widetilde{R}_j}\frac{|f_i(y)|}{\lambda}dy \bigg)^{1/m} \\
&\leq \bigg(\prod_{i=1}^m \sum_{j=1}^{\ell}\int_{\widetilde{R}_j} \frac{|f_i(y)|}{\lambda}dy \bigg)^{1/m} \\
&=\bigg(\prod_{i=1}^m \int_{\widetilde{E}} \sum_{j=1}^{\ell} \mathbf{1}_{\widetilde{R}_j}(y) \frac{|f_i(y)|}{\lambda}dy \bigg)^{1/m}.
\end{align*}
Hence, from the H\"{o}lder inequality, together with $(\ref{Holder})$ and $(\ref{Psi})$, it now follows that
\begin{align*}
1 & \leq \prod_{i=1}^m \frac{1}{|\widetilde{E}|} \int_{\widetilde{E}} \sum_{j=1}^\ell \mathbf{1}_{\widetilde{R}_j}(y) \cdot |\widetilde{E}|^{\frac{\alpha}{mn}} \frac{|f_i(y)|}{\lambda} dy \\
&\leq \prod_{i=1}^m \Big\| \sum_{j=1}^\ell \mathbf{1}_{\widetilde{R}_j} \Big\|_{\Psi_n,\widetilde{E}}  \Big\| |\widetilde{E}|^{\frac{\alpha}{mn}} \frac{f_i}{\lambda} \Big\|_{\Phi_n,\widetilde{E}} \\
&\leq \prod_{i=1}^m \delta^{-1} \Big\| |\widetilde{E}|^{\frac{\alpha}{mn}} \frac{f_i}{\lambda} \Big\|_{\Phi_n,\widetilde{E}} \\
&= \prod_{i=1}^m \Big\| \delta^{-1} |\widetilde{E}|^{\frac{\alpha}{mn}} \frac{f_i}{\lambda} \Big\|_{\Phi_n,\widetilde{E}}. \\
\end{align*}

Applying Lemma $\ref{Phi-Phi-m}$, we deduce that
$$
1 \leq \prod_{i=1}^{m}\frac{1}{|\widetilde{E}|}\int_{\widetilde{E}}\Phi_n^{(m)}
\Big(\delta^{-1}|\widetilde{E}|^{\frac{\alpha}{mn}} \frac{|f_i(y)|}{\lambda} \Big)dy.
$$
Notice that the function $\Phi_n^{(m)}$ is sub-multiplicative, we get
\begin{align}\label{cc}
1 &\lesssim \prod_{i=1}^m \frac{1}{|\widetilde{E}|} \int_{\widetilde{E}} \Phi_n^{(m)} \big(|\widetilde{E}|^{\frac{\alpha}{mn}}\big) \Phi_n^{(m)} \Big(\frac{|f_i(y)|}{\lambda} \Big) dy \\
&\nonumber\lesssim \prod_{i=1}^m \frac{1}{|\widetilde{E}|^{1-\frac{\alpha}{mn}}} \Big[1 + \big(\log^+ |\widetilde{E}|^{\frac{\alpha}{mn}} \big)^{n-1} \Big]^m \int_{\widetilde{E}} \Phi_n^{(m)} \Big(\frac{|f_i(y)|}{\lambda} \Big) dy,
\end{align}
where we have used the fact that $\Phi_{n}^{(m)}(t) \lesssim t [1+(\log^+ t)^{n-1}]^m$.
Moreover, (\ref{cc}) implies that
\begin{equation}\label{E-E}
|\widetilde{E}|^{m-\frac{\alpha}{n}}
\lesssim \prod_{i=1}^m \Big[1 + \big(\log^+ |\widetilde{E}|^{\frac{\alpha}{mn}} \big)^{n-1} \Big]^m \int_{\Rn} \Phi_n^{(m)} \Big(\frac{|f_i(y)|}{\lambda} \Big) dy.
\end{equation}
In order to get further estimate, we need a basic fact as follows: \\ if $\theta \in (0,1)$, then there exists a constant $C_0>1$ and $\beta$ small enough such that
\begin{equation}\label{fact}
0 < \beta < \frac{1-\theta}{mn}, \ \ 1 + \log^+ t^{\theta} \leq t^\beta, \ \ \text{if} \ t > C_0.
\end{equation}
If $|\widetilde{E}| > C_0$, then by the inequalities $(\ref{E-E})$ and $(\ref{fact})$ we have
$$
|\widetilde{E}|^{m-\frac{\alpha}{n}}
\lesssim |\widetilde{E}|^{m^2(n-1)\beta} \prod_{i=1}^m \int_{\Rn} \Phi_n^{(m)} \Big(\frac{|f_i(y)|}{\lambda} \Big) dy.
$$
And hence
$$
|\widetilde{E}|^{m-\frac{\alpha}{n}-m^2(n-1)\beta}
\lesssim \prod_{i=1}^m \int_{\widetilde{E}} \Phi_n^{(m)} \Big(\frac{|f(x)|}{\lambda} \Big) dx.
$$
Therefore,
$$
\log^+ |\widetilde{E}|^{\frac{\alpha}{mn}}
\lesssim \frac{\alpha}{mn} \log^+ \prod_{i=1}^m \int_{\Rn} \Phi_n^{(m)} \Big(\frac{|f_i(y)|}{\lambda} \Big) dy.
$$
From this inequality and $(\ref{E-E})$, we obtain
\begin{equation}\label{E-1}
|\widetilde{E}|^{m-\frac{\alpha}{n}}
\lesssim \prod_{i=1}^m \bigg[ 1 + \left(\frac{\alpha}{mn} \log^+ \prod_{j=1}^m \int_{\Rn} \Phi_n^{(m)} \Big(\frac{|f_j(y)|}{\lambda}\Big) dy \right)^{n-1}\bigg]^m \int_{\Rn} \Phi_n^{(m)}\left(\frac{|f_i(y)|}{\lambda}\right) dy.
\end{equation}
On the other hand, if $|\widetilde{E}| \leq C_0$, then
$$
1 + \big( \log^+ |\widetilde{E}|^{\frac{\alpha}{mn}} \big)^{n-1} \lesssim 1.
$$
Hence,
\begin{equation}\label{E-2}
|\widetilde{E}|^{m-\frac{\alpha}{n}}
\lesssim \prod_{i=1}^m \int_{\Rn} \Phi_n^{(m)} \left(\frac{|f_i(y)|}{\lambda}\right) dy.
\end{equation}
Consequently, combining $(\ref{E-1})$, $(\ref{E-2})$ with $|E| \lesssim |\widetilde{E}|$, we deduce the desired result.

\qed

Next, we will demonstrate Theorem $\ref{Theorem one-weight}$. The proof will be based on Theorem $\ref{Theorem M-B}$, which will be proved in Section $\ref{Sec-Last}$. First we recall the definition of the generalized H\"{o}lder's inequality on Orlicz spaces due to O'Neil \cite{O'Neil}.
\begin{lemma}\cite{O'Neil}\label{generalized Holder}
If $\mathcal{A}$, $\mathcal{B}$ and $\mathcal{C}$ are Young functions satisfying
$$
\mathcal{A}^{-1}(t) \mathcal{C}^{-1}(t) \leq \mathcal{B}^{-1}(t), \ \text{for any} \ t>0,
$$
then for all functions $f, g$ and any measurable set $E \subset \Rn$, the following inequality holds
\begin{equation}\label{G-H-I}
\big\| f g \big\|_{\mathcal{B},E} \leq 2 \big\| f \big\|_{\mathcal{A},E} \big\| g \big\|_{\mathcal{C},E} .
\end{equation}
\end{lemma}

\vspace{0.3cm}
\noindent\textbf{Proof of Theorem 2.3.}
The process of our proof is
$(\ref{one-weight-R-r}) \Leftrightarrow (\ref{one-weight-R})
\Rightarrow (\ref{M-R-a-Phi}) \Rightarrow (\ref{M-R-a})
\Rightarrow (\ref{one-weight-R})$. In fact,
$(\ref{one-weight-R-r}) \Leftrightarrow (\ref{one-weight-R})$ is contained in
Theorem $2.2$ \cite{CXY}. From
Lemma $\ref{generalized Holder}$, it follows that
$\mathcal{M}_{\mathcal{R},\alpha}(\vec{f})
\leq \mathcal{M}_{\mathcal{R},\alpha,\Phi_{k+1}}(\vec{f})$.
This shows $(\ref{M-R-a-Phi}) \Rightarrow (\ref{M-R-a})$. Moreover, taking
$f_i=w_i^{-p_i'} \chi_R$ for a given rectangle $R$, we may obtain
$(\ref{M-R-a}) \Rightarrow (\ref{one-weight-R})$. Hence,
it remains to prove $(\ref{one-weight-R}) \Rightarrow (\ref{M-R-a-Phi})$.

By Theorem $\ref{weight-2}$ and Theorem $6.7$ \cite[p.~458]{C-Rubio}, it is easy to see that
$\nu_{\vec{w}}^q$ satisfies the condition $(A)$ and $w_i^{-p_i'}$ satisfies
the reverse H\"{o}lder inequality. Thus, there exist constants
$c_i > 0, \ r_i > 1 \ (i=1,\cdots,m)$ such that
\begin{equation}\label{Reverse Holder}
\bigg( \frac{1}{|R|} \int_R w_i^{-p_i' r_i} dx \bigg)^{\frac{1}{r_i}} \leq \frac{c_i}{|R|} \int_R w_i^{-p_i'} dx, \text{\ for \ any\  rectangle}\ R.
\end{equation}
For fixed $k \in \N$, we introduce the notation
$$
\mathcal{A}_i(t) = t^{r_i p_i'}, \ \mathcal{C}_i(t)
=[t(1+\log ^{+}t)^k]^{(r_ip_i')'}.
$$
Then, one may obtain that
$$
\mathcal{A}_i^{-1}(t)
= t^{\frac{1}{r_i p_i'}} \
\text{and} \ \ \mathcal{A}_i^{-1}(t) \mathcal{C}_i^{-1}(t)
 \thickapprox \Phi_{k+1}^{-1}(t).
$$
Notice that $\mathcal{C}_i\in B_{p_i}^*$ and $\mathcal{C}_i$ is
submultiplicative. From the Proposition $2.2$ \cite{L-Luque}, it now follows that
$$
 M_{\mathcal{R},\mathcal{C}_i} : L^{p_i}(\Rn) \rightarrow L^{p_i}(\Rn),
\ \  \ i=1,\ldots,m.
$$
This yields immediately that
$$
\mathcal{M}_{\mathcal{R},\overrightarrow{\mathcal{C}}} : L^{p_1}(\Rn) \times
\cdots \times L^{p_m}(\Rn) \rightarrow L^p(\Rn).
$$
In addition, for a given rectangle $R$, $(\ref{Reverse Holder})$ yields that
\begin{align*}
&|R|^{\frac{\alpha}{n} + \frac{1}{q} -\frac{1}{p}}
\left( \frac{1}{|R|} \int_R \nu_{\vec{w}}^q dx \right)^{\frac{1}{q}}
\prod_{i=1}^m \big\| w_i^{-1} \big\|_{\mathcal{A}_i,{R}} \\
&= \left( \frac{1}{|R|} \int_R \nu_{\vec{w}}^q dx \right)^{\frac{1}{q}}
\prod_{i=1}^m \left( \frac{1}{|R|} \int_R w_i^{-r_i p_i'} dx
\right)^{\frac{1}{r_i p_i'}} \\
&\lesssim \left( \frac{1}{|R|} \int_R \nu_{\vec{w}}^q dx \right)^{\frac{1}{q}}
\prod_{i=1}^m \left( \frac{1}{|R|} \int_R w_i^{-p_i'} dx
\right)^{\frac{1}{p_i'}} \\
&\leq [\vec{w}]_{A_{(\vec{p},q),\mathcal{R}}} < \infty.
\end{align*}
This implies that $(\vec{w},\nu_{\vec{w}})$ satisfies the two weighted condition $(\ref{two-weight-Young condition})$.
By Theorem $\ref{Theorem M-B}$, we get
$$
\mathcal{M}_{\mathcal{R},\alpha,\Phi_{k+1}} : L^{p_1}(w_1^{p_1}) \times
\cdots \times L^{p_m}(w_m^{p_m}) \rightarrow L^q(\nu_{\vec{w}}^q).
$$

Therefore, in all, we have completed the proof of Theorem $\ref{Theorem one-weight}$.
\qed

\section{Proofs of the rest Theorems and Corollaries}\label{Sec-Last}

To prove Theorem $\ref{Theorem M-B}$, we first introduce the definition of the general basis and a key covering lemma.
\begin{definition}[\cite{J}, \cite{JT}]
Let $\mathscr{B}$ be a basis and let $0 < \alpha < 1$. A finite sequence $\{ A_i \}_{i=1}^N \subset \mathscr{B}$ of
sets of finite $dx$-measure is called $\alpha$-scattered with respect to the Lebesgue measure if
$$
\Big| A_i \bigcap \bigcup_{s<i} A_s \Big| \leq \alpha |A_i|, \ \ \text{for all} \ \ 1 < i \leq N .
$$
\end{definition}

\begin{lemma}[\cite{GLPT}, \cite{J}]\label{Lemma condition (A)}
Let $\mathscr{B}$ be a basis and let $w$ be a weight associated to this basis. Suppose further that $M_{\mathscr{B}}$ satisfies condition $(A_{\mathscr{B},\gamma,w})$ for some $0 < \gamma < 1$. Then, given any finite sequence $\{A_i\}_{i=1}^N$ of sets $A_i \in \mathscr{B}$, one can find a subsequence $\{ \tilde{A}_i \}_{i \in I}$ such that
\begin{enumerate}
\item [(a)] $\{ \tilde{A}_i \}_{i \in I}$ is $\gamma$-scattered with respect to the Lebesgue measure;
\item [(b)] $\tilde{A}_i = A_i, i \in I$;
\item [(c)] for any $1 \leq i < j \leq N + 1$,
$$
w \Big( \bigcup_{s < j} A_s \Big) \lesssim  w \Big( \bigcup_{s < i} A_s \Big) + w \Big( \bigcup_{i \leq s < j} \tilde{A}_s \Big),
$$
where $\tilde{A}_s = \emptyset$ when $s \not\in I$.
\end{enumerate}
\end{lemma}

\vspace{0.3cm}
\noindent\textbf{Proof of Theorem 2.6.}
The idea of the following arguments is essentially a combination of the ideas from \cite{GLPT}, \cite{J}, \cite{L-Luque}.
Let $N > 0$ be a large integer. We will prove the required estimate for the quantity
$$ \int_{\{2^{-N} < \mathcal{M}_{\mathscr{B},\varphi,\overrightarrow{\Psi}}(\vec{f}) \leq 2^{N+1}\}} \mathcal{M}_{\mathscr{B},\varphi,\overrightarrow{\Psi}}(\vec{f})(x)^q v^q dx ,$$
with a bound independent of $N$.
We begin with the following claim.
\begin{claim}
For each integer $k$ with $|k| \leq N$, there exists a compact set $K_k$ and a finite sequence $b_k = \{ B_r^k \}_{r \geq 1}$ of sets $B_r^k \in \mathscr{B}$ such that
$$ v^q(K_k) \leq v^q(\{\mathcal{M}_{\mathscr{B},\varphi,\overrightarrow{\Psi}}(\vec{f}) > 2^k \}) \leq 2 v^q(K_k) .$$
Moreover, $\{ \cup_{B \in b_k} B\}_{k={-N}}^N$ is decreasing and there holds that
$$
\bigcup_{B \in b_k} B \subset K_k \subset \{\mathcal{M}_{\mathscr{B},\varphi,\overrightarrow{\Psi}}(\vec{f}) > 2^k \}
$$
and
\begin{equation}\label{B-r-k}
\varphi(|B_r^k|) \prod_{j=1}^m \big\| f_j \big\|_{\Psi_j,B_r^k} > 2^k.
\end{equation}
\end{claim}
\begin{proof}
For each $k$ we choose a compact set $\widetilde{K}_k \subset \{\mathcal{M}_{\mathscr{B},\varphi,\overrightarrow{\Psi}}(\vec{f}) > 2^k \}$ such that $$ v^q(\widetilde{K}_k) \leq v^q(\{\mathcal{M}_{\mathscr{B},\varphi,\overrightarrow{\Psi}}(\vec{f}) > 2^k \}) 2 \leq v^q(\widetilde{K}_k) .$$ For this $\widetilde{K}_k$, there exists a finite sequence $b_k = \{B^k_r \}_{r \geq 1}$ of sets $B^k_r \in \mathscr{B}$ such
that every $B^k_r$ satisfies $(\ref{B-r-k})$ and such that
$\widetilde{K}_k \subset \cup_{B \in b_k}B \subset \{\mathcal{M}_{\mathscr{B},\varphi,\overrightarrow{\Psi}}(\vec{f}) > 2^k \}$.
Now, we take a compact set $K_k$ such that
$\cup_{B \in b_k}B \subset K_k \subset \{\mathcal{M}_{\mathscr{B},\varphi,\overrightarrow{\Psi}}(\vec{f}) > 2^k \}$. Finally, to ensure that $\{\cup_{B \in b_k} B\}_{k=-N}^N$ is decreasing, we begin the above selection from $k = N$ and once a selection is done for $k$ we do the selection for $k-1$ with the next additional requirement
$\widetilde{K}_{k-1} \supset K_k$. This finishes the proof of the claim.
\end{proof}
We continue with the proof of Theorem 2.4. Since $\{\cup_{B \in b_k} B\}_{k=-N}^N$ is a sequence of decreasing sets, we set
\begin{equation*}
\Omega_k =
\begin{cases}
\cup_r B^k_r = \cup_{B \in b_k} B \ \ \ &\text{when  } |k| \leq N . \\
\emptyset   \ \ \ &\text{otherwise } .
\end{cases}
\end{equation*}
Then, these sets are decreasing in $k$, i.e.,$ \Omega_{k+1} \subset \Omega_k$ when $-N < k \leq N$.

We now distribute the sets in $\cup_k b_k$ over $\mu$ sequences $\{A_i(\ell)\}_{i \geq 1}$, $0 \leq \ell \leq \mu - 1$,
where $\mu$ will be chosen momentarily to be an appropriately large natural number. Set $i_0(0) = 1$. In the first
$i_1(0) - i_0(0)$ entries of $\{A_i(0)\}_{i \geq 1}$, i.e., for $$i_1(0) \leq i < i_1(0),$$
we place the elements of the sequence $b_N = \{ B_r^N \}_{r \geq 1} $ in the order indicated by the index $r$.
For the next $i_2(0) - i_1(0)$ entries of $\{A_i(0)\}_{i \geq 1}$, i.e., for $$i_1(0) \leq i < i_2(0),$$
we place the elements of the sequence $b_{N-\mu}$. Continue in this way until we reach the first integer $m_0$ such that
$N - m_0 \mu \geq -N$, when we stop. For indices $i$ satisfying  $$i_{m_0}(0) \leq i < i_{m_0 + 1}(0),$$
we place in the sequence $\{A_i(0)\}_{i \geq 1}$ the elements of $b_{N - m_0 \mu}$. The sequences $\{A_i(\ell)\}_{i \geq 1}$,
$1 \leq \ell \leq \mu - 1$, are defined similarly, starting from $b_{N-\ell}$ and using the families $b_{N-\ell-s\mu}$,
$s = 0, 1, ¡¤ ¡¤ ¡¤ ,m_l$, where $m_l$ is chosen to be the biggest integer such that $N - l - m_l \mu \geq - N$.

Since $v^q$ is a weight associated to $\mathscr{B}$ and it satisfies the condition $(A)$, we can apply Lemma
$\ref{Lemma condition (A)}$ to each $\{A_i(\ell)\}_{i \geq 1}$ for some fixed $0 < \lambda < 1$. Then we obtain sequences
$$\{\tilde{A}_i(\ell)\}_{i \geq 1} \subset \{A_i(\ell)\}_{i \geq 1} , \ 0 \leq \ell \leq \mu - 1,$$ which are $\lambda$-scattered with respect to the Lebesgue measure. In view of the definition of the set $k$ and the construction of the families $\{A_i(\ell)\}_{i \geq 1}$, we may use
assertion $(c)$ of Lemma $\ref{Lemma condition (A)}$ to show that: for any $k = N - \ell - s\mu$ with $0 \leq \ell \leq \mu -1$ and $1 \leq s \leq m_{\ell}$, it yields that
\begin{align*}
v^q(\Omega_k) = v^q(\Omega_{N-\ell-s\mu}) &\lesssim  v^q(\Omega_{k+\mu}) + v^q \bigg( \bigcup_{i_s(\ell) \leq i \leq i_{s+1}(\ell)} \tilde{A}_i(\ell) \bigg)  \\
&\leq  v^q(\Omega_{k+\mu}) + \sum_{i=i_s(\ell)}^{i_{s+1}(\ell)-1} v^q(\tilde{A}_i(\ell)).
\end{align*}
For the case $s = 0$, we have $k = N - \ell$ and
\begin{align*}
v^q(\Omega_k) = v^q(\Omega_{N-\ell})
&\lesssim \sum_{i=i_0(\ell)}^{i_{1}(\ell)-1} v^q(\tilde{A}_i(\ell)).
\end{align*}
Now, all these sets $\{ \tilde{A}_i(\ell) \}_{i=i_s(\ell)}^{i_{s+1}(\ell)}$ belong to $b_k$ with $k = N - \ell - s\mu$ and then
\begin{equation}\label{A-i-l}
\varphi(|\tilde{A}_i(\ell)|) \prod_{j=1}^m \big\| f_j \big\|_{\Psi_j,\tilde{A}_i(\ell)} > 2^k.
\end{equation}
Therefore, it now readily follows that
$$
\int_{\{2^{-N} < \mathcal{M}_{\mathscr{B},\varphi,\overrightarrow{\Psi}}(\vec{f}) \leq 2^{N+1}\}} \mathcal{M}_{\mathscr{B},\varphi,\overrightarrow{\Psi}}(\vec{f})(x)^q v^q dx
\lesssim \sum_{k=-N}^{N-1} 2^{kq} v^q(\Omega_k) := \Delta_1.
$$
Thus, we have
\begin{equation}\aligned\label{sum k-N}
\Delta_1
&=\sum_{\ell=0}^{\mu-1} \sum_{0 \leq s \leq m_{\ell}} 2^{q(N-\ell-s\mu)} v^q(\Omega_{N-\ell-s\mu})\\
&\lesssim \sum_{\ell=0}^{\mu-1} \sum_{0 \leq s \leq m_{\ell}} 2^{q(N-\ell-s\mu)} v^q(\Omega_{N-\ell-s\mu+\mu}) \\
&\quad + \sum_{\ell=0}^{\mu-1} \sum_{0 \leq s \leq m_{\ell}} 2^{q(N-\ell-s\mu)} \sum_{i=i_s(\ell)}^{i_{s+1}(\ell)-1} v^q(\tilde{A}_i(\ell))\\
& := \Delta_2 + \Delta_3.
\endaligned
\end{equation}
To analyze the contribution of $\Delta_2$, we choose $\mu$ so large that $2^{-q \mu} \leq \frac12$. Therefore,
\begin{equation}\label{Delta-3}
\Delta_2 = 2^{-q \mu}\sum_{\ell=0}^{\mu-1} \sum_{0 \leq s \leq {m_{\ell-1}}} 2^{q(N-\ell-s\mu)} v^q(\Omega_{N-\ell-s\mu})
\leq 2^{-q \mu} \sum_{k=-N}^{N-1} 2^{kq} v^q(\Omega_k) \leq \frac12 \Delta_1.
\end{equation}
Since everything involved is finite, $\Delta_2$ can be subtracted from $\Delta_1$. This yields that
$$ \int_{\{2^{-N} < \mathcal{M}_{\mathscr{B},\varphi,\overrightarrow{\Psi}}(\vec{f}) \leq 2^{N+1}\}} \mathcal{M}_{\mathscr{B},\varphi,\overrightarrow{\Psi}}(\vec{f})(x)^q v^q dx
\lesssim \Delta_1 \lesssim \Delta_3.$$

Next we consider the contribution of $\Delta_3$.
For the sake of simplicity, for each $\ell$ we let $I(\ell)$ be the index set of
$ \{ \tilde{A}_i(\ell) \}_{0 \leq s \leq m_{\ell}, i_s(\ell) \leq i < i_{s+1}(\ell)}$.
By $(\ref{A-i-l})$ and the generalized H\"{o}lder¡¯s inequality $(\ref{G-H-I})$, we obtain
\begin{equation}\aligned\label{sum-sum-sum}
\Delta_3 &\lesssim \sum_{\ell=0}^{\mu-1} \sum_{i \in I(\ell)} v^q(\tilde{A}_i(\ell)) \bigg[ \varphi(|\tilde{A}_i(\ell)|) \prod_{i=1}^m \big\| f_i \big\|_{\Phi,\tilde{A}_i(\ell)} \bigg]^q \\
&\lesssim \sum_{l=0}^{\mu-1} \sum_{i \in I(\ell)} \bigg[ \prod_{j=1}^m \big\| f_j \big\|_{C_j,\tilde{A}_i(\ell)}^p  |\tilde{A}_i(\ell)| \bigg]^{q/p} \\
&\quad \times \bigg[ \varphi(|\tilde{A}_i(\ell)|)|\tilde{A}_i(\ell)|^{\frac{1}{q} -\frac{1}{p}} \left( \frac{1}{|\tilde{A}_i(\ell)|} \int_{\tilde{A}_i(\ell)} v^q dx \right)^{\frac{1}{q}} \prod_{j=1}^m \big\| w_j^{-1} \big\|_{A_j,\tilde{A}_i(\ell)} \bigg]^q \\
& \lesssim \sum_{\ell=0}^{\mu-1} \sum_{i \in I(\ell)}
\bigg[ \prod_{j=1}^m \big\| f_j w_j \big\|_{C_j, \tilde{A}_i(\ell)}^p |\tilde{A}_i(\ell)| \bigg]^{q/p} \\
&\leq \bigg[ \sum_{\ell=0}^{\mu-1} \sum_{i \in I(\ell)}
 \prod_{j=1}^m \big\| f_j w_j \big\|_{C_j, \tilde{A}_i(\ell)}^p |\tilde{A}_i(\ell)| \bigg]^{q/p} ,
\endaligned
\end{equation}
where in the third step we used the two-weight condition $(\ref{two-weight-Young condition})$.

Now, we introduce the notations
\begin{equation}\label{notations}
E_1(\ell) = \tilde{A}_i(\ell) \ \text{ and } \  E_i(\ell) = \tilde{A}_i(\ell) \setminus \bigcup_{s<i} \tilde{A}_s(\ell),
\ \ \ \forall i \in I(\ell).
\end{equation}
Since the sequences $\{ \tilde{A}_i(\ell) \}_{i \in I(\ell)}$ are $\lambda$-scattered with respect to the Lebesgue measure,
$|\tilde{A}_i(\ell)| \leq \frac{1}{1 - \lambda} |E_i(\ell)|$ for each $i$. Then we have the following estimate for $(\ref{sum-sum-sum})$,
\begin{equation}\label{C-1-lambda}
\Delta_3 \lesssim \bigg[ \frac{1}{1 - \lambda} \sum_{\ell=0}^{\mu-1} \sum_{i \in I(\ell)} \prod_{j=1}^m \big\| f_j w_j \big\|_{C_j,\tilde{A}_i(\ell)}^p |E_i(\ell)| \bigg]^{q/p} \\.
\end{equation}
The collection $\{ E_i(\ell) \}_{i \in I(\ell)}$ is a disjoint family, we can therefore use the fact that
$\mathcal{M}_{\mathscr{B}, \overrightarrow{C}}$ is bounded from $L^{p_1}(\Rn)\times \cdots \times L^{p_m}(\Rn)$ to $L^p(\Rn)$, to estimate the inequality $(\ref{C-1-lambda})$. Hence
\begin{align*}
\Delta_3 & \lesssim \bigg[\sum_{\ell=0}^{\mu-1} \sum_{i \in I(\ell)} \int_{E_i(\ell)} \Big( \mathcal{M}_{\mathscr{B}, \overrightarrow{C}}(f_1 w_1, \cdots, f_m w_m)(x) \Big)^p dx \bigg]^{q/p} \\
&\lesssim \bigg[\int_{\Rn} \Big(\mathcal{M}_{\mathscr{B},\overrightarrow{C}}(f_1 w_1, \cdots, f_m w_m)(x) \Big)^p dx \bigg]^{q/p} \\
&\lesssim \prod_{i=1}^m \big\| f_i w_i \big\|_{L^{p_i}(\Rn)}^q
=\prod_{i=1}^m \big\| f_i \big\|_{L^{p_i}(w_i^{p_i})}^q.
\end{align*}
Finally, letting $N \rightarrow \infty$, we finish the proof.
\qed

\vspace{0.5cm}
\noindent\textbf{Proof of Corollary 2.7.}
For each $i=1, \cdots, m$, we set $\widetilde{w}_i := w_i^{1/{p_i}}$ and $\Psi_i(t) := t^{{p_i}'r}$ for any $t > 0$. Set
$\tilde{v} := v^{1/q}$. Then the power bump condition $(\ref{power bump condition})$ can be rewritten as
\begin{equation*}
\sup_{B\in \mathscr{B}} |B|^{\frac{\alpha}{n} + \frac{1}{q} -\frac{1}{p}} \left( \frac{1}{|B|} \int_B \widetilde{v}^q dx \right)^{\frac{1}{q}} \prod_{i=1}^m \big\| \widetilde{w}_i^{-1} \big\|_{\Psi_i,B} < \infty.
\end{equation*}
In this case, for all $x \in \Rn$,
$$ M_{\mathscr{B},\bar{\Phi}_i}f(x) = \sup_{\substack{B \ni x \\ B \in \mathscr{B}}} \Big\{ \frac{1}{|B|} \int_B |f(y)|^{(p_i' r)'} dy \Big\}^{1/{(p'_i r)'}}.$$
Since $\mathscr{B}$ is a Muckenhoupt basis and $(p_i'r)' < p_i$, every $ M_{\mathscr{B},\bar{\Psi}_i}$ is bounded on $L^{p_i}(\Rn)$. It is easy to see that
$$ \mathcal{M}_{\mathscr{B}, \overrightarrow{\bar{\Psi}}}(\vec{f})(x) \leq \prod_{i=1}^m  M_{\mathscr{B},\bar{\Psi}_i}(f_i)(x),\ x \in \Rn. $$
This inequality implies that $ \mathcal{M}_{\mathscr{B},\overrightarrow{\bar{\Psi}}}$ is bounded from $L^{p_1}(\Rn) \times \cdots \times L^{p_m}(\Rn)$ to $L^p(\Rn)$. Thus, from Theorem $\ref{Theorem M-B}$, it follows that
$$\mathcal{M}_{\mathscr{B},\alpha} : L^{p_1}(\tilde{w}_1^{p_1}) \times \cdots \times L^{p_m}(\tilde{w}_m^{p_m}) \rightarrow L^q(\tilde{v}^q),$$
which completes the proof.
\qed

\vspace{0.5cm}
\noindent\textbf{Proof of Theorem 2.4.}
The fact that $\mathcal{R}$ is a Muckenhoupt basis can be found in \cite[p.~454]{C-Rubio}. Moreover, for the rectangle family $\mathcal{R}$, the $A_{\infty,\mathcal{R}}$ condition is equivalent to Tauberian condition $(A_{\mathcal{R},\gamma,w})$, which was proved in Corollary $4.8$ \cite{HLP}. Therefore, Theorem $\ref{Theorem power-R}$ follows from these facts and Corollary $\ref{Theorem power-B}$.
\qed

\vspace{0.3cm}
\noindent\textbf{Proof of Corollary 2.5.}
From Theorem $\ref{weight-1}$, it follows that $v^p \in A_{mp,\mathcal{R}} \subset A_{\infty,\mathcal{R}}$. As for $v=\prod_{i=1}^m u_i^{1/{p_i}}$ and $w_i=M_{\alpha p_i/m}(u_i)$, it is easy to verify that $(\vec{w},v)$ satisfies the power bump condition $(\ref{power condition})$. Hence, it yields the desired result.
\qed

\vspace{0.3cm}
\noindent\textbf{Proof of Theorem 2.8.}
Proposition $\ref{Theorem M-Y}$ follows by using the similar arguments as we have done in the proof of Theorem $\ref{Theorem M-B}$. The difference lies in the boundedness of $\mathcal{M}_{\vec{Y}'}$ and the generalized H\"{o}lder's inequality
$$
\int_{\Rn}|f(x)g(x)|dx \leq ||f||_{X} ||g||_{X'},
$$
for any Banach function space $X$.
\qed

\vspace{0.5cm}
\noindent\textbf{Proof of Theorem 2.2.}
It is well known that there exists some $h \in L^{r'}(\Rn)$ with norm $||h||_{L^{r'}(\Rn)}=1$ such that
$$
\big\| M_{\mathcal{R}} f \big\|_{L^p(\ell^q,w^p)}^p
= \int_{\Rn} \Big( \sum_{j}M_{\mathcal{R}} f_j(x)^q w(x)^q \Big)^{r} dx
=\sum_j \int_{\Rn} M_{\mathcal{R}} f_j(x)^q w(x)^q h(x) dx.
$$
In order to estimate $\int_{\Rn} M_{\mathcal{R}} f_j(x)^q w(x)^q h(x) dx$ for fixed $j$, we adopt the similar method as in the proof of Theorem $\ref{Theorem M-B}$. Since we obtained the inequality $(\ref{Delta-3})$, we get for any fixed $N>0$
\begin{equation*}\aligned
\Lambda_{j,N}&:=\int_{\{2^{-N} < M_{\mathcal{R}}f_j(x) \leq 2^{N+1}\}} M_{\mathcal{R}}f_j(x)^q w(x)^q h(x) dx \\
&\lesssim \sum_{\ell=0}^{\mu-1} \sum_{0 \leq s \leq m_{\ell}} \sum_{i=i_s(\ell)}^{i_{s+1}(\ell)-1} (w^q h)(\tilde{A}_i(\ell))
\bigg( \frac{1}{|\tilde{A}_i(\ell)|} \int_{\tilde{A}_i(\ell)} |f_j(x)| dx \bigg)^q.
\endaligned
\end{equation*}
Making use of the H\"{o}lder inequality and two weight condition $(\ref{A-B})$, we deduce that
\begin{align*}
\Lambda_{j,N}
&\lesssim  \sum_{\ell,s,i} || w^q ||_{A,\tilde{A}_i(\ell)}
||h||_{\bar{A},\tilde{A}_i(\ell)} ||f_j v ||_{\bar{B},\tilde{A}_i(\ell)}^q ||v^{-1}||_{B,\tilde{A}_i(\ell)}^q |\tilde{A}_i(\ell)| \\
&\lesssim \sum_{\ell,s,i} ||f_j v ||_{\bar{B},\tilde{A}_i(\ell)}^q ||h||_{\bar{A},\tilde{A}_i(\ell)}  |\tilde{A}_i(\ell)| .
\end{align*}
Using the same notations $\{E_i(\ell)\}$ as in $(\ref{notations})$, we may obtain
\begin{align*}
\Lambda_{j,N}
&\lesssim \sum_{\ell,s,i} ||f_j v ||_{\bar{B},\tilde{A}_i(\ell)}^q ||h||_{\bar{A},\tilde{A}_i(\ell)}  |E_i(\ell)|  \\
&\leq \sum_{\ell,i}\int_{E_i(\ell)} M_{\mathcal{R},\bar{B}}(f_j v)(x)^q \ M_{\mathcal{R},\bar{A}}h(x) \ dx \\
&\lesssim \int_{\Rn} M_{\mathcal{R},\bar{B}}(f_j v)(x)^q \ M_{\mathcal{R},\bar{A}}h(x) \ dx.
\end{align*}
Letting $N \rightarrow \infty$, we have
$$
\int_{\Rn} M_{\mathcal{R}} f_j(x)^q w(x)^q h(x) dx
\lesssim \int_{\Rn} M_{\mathcal{R},\bar{B}}(f_j v)(x)^q \ M_{\mathcal{R},\bar{A}}h(x) \ dx.
$$
Therefore, from the H\"{o}lder inequality and the following Proposition $\ref{pro-vector}$, it follows that
\begin{align*}
\big\| M_{\mathcal{R}} f \big\|_{L^p(\ell^q,w^p)}^q
&\lesssim \Big\| \Big(\sum_{j=1}( M_{\mathcal{R},\bar{B}}(f_j v))^q \Big)^{1/q} \Big\|_{L^p(\Rn)}^q \
\big\| M_{\mathcal{R},\bar{A}}h \big\|_{L^{r'}(\Rn)} \\
&\lesssim \Big\| \Big(\sum_{j=1}(f_j v)^q \Big)^{1/q} \Big\|_{L^p(\Rn)}^q \ \big\| h \big\|_{L^{r'}(\Rn)}
= ||f||_{L^p(\ell^q,v^p)}^q.
\end{align*}
This completes the proof of Theorem $\ref{Theorem lq}$.

\qed

It remains to show the following proposition.
\begin{proposition}\label{pro-vector}
Let $1 < q < p < \infty$. Suppose that $\Phi$ is a Young function such that
$\Phi \in B_q^*$. If $(A_{\mathcal{R},\gamma,g})$ condition holds for some
fixed $\gamma \in (0,1)$ and any nonnegative function $g \in L^{r'}(\Rn)$
with $||g||_{L^{r'}(\Rn)}=1$, then we have
$$
\big\| M_{\mathcal{R},\Phi}f \big\|_{L^p(\ell^q,\Rn)}
\lesssim \big\| f \big\|_{L^p(\ell^q,\Rn)}.
$$
\end{proposition}

\begin{proof}
Set $r=p/q$. Then, it holds that
$$
\big\| M_{\mathcal{R},\Phi}f \big\|_{L^p(\ell^q,\Rn)}^q
=\sup_{||g||_{L^{r'}(\Rn)}=1}
\bigg| \int_{\Rn} \sum_j M_{\mathcal{R},\Phi}f_j(x)^q g(x) dx \bigg|.
$$
For fixed $g \in L^{r'}(\Rn)$ with $||g||_{L^{r'}(\Rn)}=1$, by the
Fefferman-Stein inequality for the maximal operator $M_{\mathcal{R},\Phi}$
( see Theorem $2.1$ \cite{L-Luque}), we have
\begin{align*}
\bigg| \int_{\Rn} \sum_j M_{\mathcal{R},\Phi}f_j(x)^q g(x) dx \bigg|
&\leq \sum_j \int_{\Rn}  M_{\mathcal{R},\Phi}f_j(x)^q |g(x)| dx \\
&\lesssim \sum_j \int_{\Rn} |f_j(x)|^q M_{\mathcal{R}} g(x) dx \\
&\leq \Big\| \sum_j |f_j|^q \Big\|_{L^r(\Rn)}
\big\| M_{\mathcal{R}} g \big\|_{L^{r'}(\Rn)} \\
&\lesssim || f ||_{L^p(\ell^q,\Rn)}^q  || g ||_{L^{r'}(\Rn)}
=|| f ||_{L^p(\ell^q,\Rn)}^q .
\end{align*}
This completes the proof of Proposition \ref{pro-vector}.
\end{proof}

\end{document}